\documentclass[10pt]{amsart}

\title{The representation theory
of\\Sylow 2-subgroups of symmetric groups}

\author{
  S Narayanan
  }
\address{
  The Institute of Mathematical Sciences (HBNI), Chennai
  }
\email{sridharn@imsc.res.in}
  
\date{\today}
\usepackage{caption}
\usepackage{filecontents}
\usepackage{amsmath, amssymb, hyperref,amsthm, graphicx, mathtools,,scalerel}
\usepackage[table]{xcolor}
\usepackage[section]{placeins}
\usepackage{tikz}
\usepackage{cite}
\usepackage{float}
\usepackage{inputenc}
\usepackage[export]{adjustbox}
\usetikzlibrary{shapes.geometric}
\usepackage[tight,footnotesize]{subfigure}
\usepackage{etoolbox}
\usepackage{ytableau}
\newcommand{\zerodisplayskips}{%
  \setlength{\abovedisplayskip}{0pt}%
  \setlength{\belowdisplayskip}{0pt}%
  \setlength{\abovedisplayshortskip}{0pt}%
  \setlength{\belowdisplayshortskip}{0pt}}
\appto{\normalsize}{\zerodisplayskips}
\appto{\small}{\zerodisplayskips}
\appto{\footnotesize}{\zerodisplayskips}

\newtheorem{prop}{Proposition}[section]

\newtheorem{lemma}{Lemma}[section]
\newtheorem{definition}{Definition}[section]

\newtheorem{theorem}{Theorem}[section]
\newtheorem{corollary}[theorem]{Corollary}
\theoremstyle{remark}
\newtheorem*{remark}{Remark}

\newcommand{\Dim}{\mathrm{dim}}
\newcommand{\Or}{\mathrm{\rho}}
\newcommand{\Extp}{\mathrm{Ext^+}}
\newcommand{\Extn}{\mathrm{Ext^-}}
\newcommand{\Extpn}{\mathrm{Ext^{\pm}}}
\newcommand{\Id}{\mathrm{Id}}
\newcommand{\Ind}{\mathrm{Ind}}
\newcommand{\Res}{\mathrm{Res}}
\newcommand{\tikznode}[3][inner sep=0pt]{\tikz[remember
picture,baseline=(#2.base)]{\node(#2)[#1]{$#3$};}}

\begin{document}
\ytableausetup{centertableaux}

\begin{abstract}
We use binary trees to study the Bratteli diagram of Sylow 2-subgroups of symmetric groups. We show that it is simple, has a recursive structure, and self-similarities at all scales. We contrast its subgraph of one-dimensional representations
with the Macdonald tree. We exploit the recursive structure to find the multiplicities of irreducible characters in the restriction to a Sylow 2-subgroup of odd-dimensional representations of the symmetric group $S_{2^k}$.
\end{abstract}

\maketitle

\section{Introduction}
\label{sec:intro}

Given a finite group $G$, a prime integer $p$ and a Sylow p-subgroup $P$ of $G$, the McKay conjecture states that the number of irreducible representations of $G$ with degree coprime to $p$ is equal to the number of irreducible representations of the normaliser of $P$ in $G$ whose degrees are coprime to $p$. The conjecture was proved for the family of symmetric groups by Macdonald in \cite{MR289677}, and for arbitrary groups when $p=2$ by Sp{\"a}th and Malle in \cite{MR3549625}. When $p=2$ and $G$ is the symmetric group $S_n$, the Sylow subgroup $P$ is self-normalising. Thus we know that there are as many odd-dimensional representations of $S_n$ as there are one-dimensional representations of a Sylow 2-subgroup of $S_n$.

Odd-dimensional irreducible representations of symmetric groups were studied by Ayyer, Prasad and Spallone in \cite{MR3510808}. In particular, it is known (see \cite[Theorem 1]{MR3510808}) that the subgraph of the Young graph comprising odd-dimensional representations of $S_n$ is a rooted binary tree that branches at every even level. This tree is called the Macdonald tree. 

As we shall see in this work, the branching graph of one-dimensional representations of Sylow 2-subgroups of symmetric groups is also a rooted binary tree that branches at every even level. In uncovering this structure we offer an overview of the character theory of these subgroups. This theory is recursive and self-similar, making the association to binary trees in  \cite[Sequence A006893]{oeis} natural. Here we describe branching rules for these characters as combinatorial operations on these `tree-like' objects. We compare the aforementioned subgraph to the Macdonald tree from \cite{MR3510808}. We also study the restriction of representations of the symmetric group $S_n$ to a Sylow 2-subgroup, and when $n$ is a power of $2$, we obtain the multiplicities of every irreducible representation of the Sylow 2-subgroup occuring in the restriction of an odd-dimensional representation (i.e., representations corresponding to hook partitions) of $S_n$. 

%fix this
The main results of this manuscript are the following:
\begin{itemize} 
\item  Theorem \ref{th:mainA}, which describes the branching rules for a family of Sylow 2-subgroups as combinatorial operations on forests of binary trees\footnote{see Definition \ref{def:for} for the definition of a forest of trees.}. As a result of this we observe a self-similarity in the Bratteli diagram of this family of subgroups. 
\item Theorem \ref{th:mainB}, which rephrases the branching rules for one-dimensional irreducible representations of these sugroups as operations on sequences of binary strings. This allows us to define this subgraph recursively. It has the structure of a binary tree, but is not isomorphic to the Macdonald tree.
\item Theorem \ref{th:mainC}, which provides a recursive formula for the multiplicity of an irreducible representation of a Sylow 2-subgroup of $S_{2^k}$ in the restriction of an odd-dimensional representation of $S_{2^k}$ to this subgroup.
\end{itemize}

Section \ref{sec:con_not} introduces certain preliminaries and notation. The first subsection is a description of the conjugacy classes and irreducible representations of Sylow 2-subgroups of $S_n$, while the second and third subsections serve as a primer on binary trees and introduce a bijection between representations (and conjugacy classes) of a Sylow 2-subgroup and forests of binary trees. The final subsection contains the definitions and concepts related to hook partitions, which are required in Section \ref{sec:res}. Section \ref{sec:brat} describes the Bratteli diagram for a family of Sylow 2-subgroups of $S_n$. In the next section we focus on the subgraph of one-dimensional representations in this Bratteli diagram. Section \ref{sec:res} focuses on the restriction of irreducible representations of a symmetric group to a Sylow 2-subgroup, which facilitates the explicit description of the bijection in  \cite[Theorem 1.1]{MR3687936}. In the final section we define generating functions for the dimensions of irreducible representations, and for the sizes of the conjugacy classes of these subgroups. 

A Sage implementation can be found at: 

https://github.com/sridharpn/2-Sylow.

\section{Preliminaries and Notation}
\label{sec:con_not}
Throughout this paper, $n$ is a positive integer with the binary expansion $$n=2^{k_1}+\dotsb+2^{k_s},$$ with $k_1 > \dotsb >k_s$. 
\begin{definition}
\label{def:bin_digits} 
The binary digits of $n$, denoted $Bin(n)$ is the set $\{k_1, \dotsc ,k_s\}$.
\end{definition}

Sylow 2-subgroups of $S_n$ are denoted $P_n$, and when $n=2^k$ for a nonnegative integer $k$, the Sylow 2-subgroup is denoted by $H_k$. 
 
\subsection{Structure and representation theory of $P_n$}

The structure of Sylow p-subgroups is well studied (see \cite{MR0028834}). It is known that:
\begin{displaymath}
P_n= \prod_{k \in Bin(n)}H_k.
\end{displaymath}  
It is also known that $H_k\cong H_{k-1} \wr C_2$, where $C_2$ is the cyclic group of order 2. An element of this wreath product is denoted $(\sigma_1,\sigma_2)^{\epsilon}$, where $\sigma_1,\sigma_2 \in H_{k-1}$ and $\epsilon \in C_2={\pm 1}$. The identity element of the group is denoted $\Id$. Multiplication is defined as follows:
\begin{displaymath}
(\sigma_1,\sigma_2)^{\epsilon_1}(\tau_1,\tau_2)^{\epsilon_2}=
\begin{cases}
(\sigma_1\tau_1,\sigma_2\tau_2)^{\epsilon_1\epsilon_2} & {\epsilon_1=1},\\
(\sigma_1\tau_2,\sigma_2\tau_1)^{\epsilon_1\epsilon_2} & {\epsilon_1=-1}.
\end{cases}
\end{displaymath}

Elementary computations yield the three types of conjugacy classes of $H_k$.
\begin{definition}
\label{def:conj_type}
Given an element $\sigma$, let $[\sigma]$ denote its conjugacy class. Then we have:
\begin{itemize}
\item $[\sigma]$ is of Type I if $$[\sigma]=[(\sigma_1,\sigma_1)^1],$$ for an element $\sigma_1 \in H_{k-1}$.
\item $[\sigma]$ is of Type II if $$[\sigma]=[(Id,\sigma_1)^{-1}],$$ for elements $\sigma_1 \in H_{k-1}$.
\item $[\sigma]$ is of Type III if $$[\sigma]=[(\sigma_1,\sigma_2)^1],$$ for elements $\sigma_1,\sigma_2 \in H_{k-1}$ and $[\sigma_1] \neq [\sigma_2]$.
\end{itemize}
\end{definition}
The cardinalities of the above listed classes (denoted $c_k([\sigma])$) and the number of classes of each are listed in Table \ref{table:conj_class}. The total number of conjugacy classes of the group $H_k$ is denoted $C_k$ in this table.

\begin{table}[h]
\begin{center}
\begin{tabular}{|c|c|c|c| } 

 \hline

 Type & Representative & \# classes & Size of class($c_k$) \\ 

 \hline
 I & $[(\sigma,\sigma)^1]$ & $C_{k-1}$ & $c_{k-1}([\sigma])^2$ \\  
 \hline
 II & $[(\Id,\sigma)^{-1} ]$ & $C_{k-1}$ & $|H_{k-1}|c_{k-1}([\sigma])$\\
 \hline
 III & $[(\sigma_1,\sigma_2)^1]$ & $C_{k-1} \choose 2$ &
$2c_{k-1}([\sigma_1])c_{k-1}([\sigma_2])$ \\ 

 \hline

\end{tabular}
\caption{Conjugacy classes of $H_k$}
\label{table:conj_class}
\end{center}
\end{table}

The enumeration of characters of Sylow 2-subgroups is a particular instance of characters of wreath products; we refer the reader to  \cite{MR1503352} and \cite{MR409621} for details on the methods used. According to \cite{MR1503352}, all irreducible representatives of $H_k$ are obtained as constituents in the induction of irreducible representations from the normal subgroup $H_{k-1} \times H_{k-1}$ to $H_k$. The irreducible representations of $H_{k-1} \times H_{k-1}$ are tensor products of two irreducible representations of $H_{k-1}$.

Let $\phi_1$ and $\phi_2$ be irreducible representations of $H_{k-1}$. If $\phi_2$ is not isomorphic to $\phi_1$, then $\Ind_{H_{k-1} \times H_{k-1}}^{H_k}(\phi_1 \otimes \phi_2)$ is an irreducible representation of $H_k$. We denote it $\Ind(\phi_1,\phi_2)$. The character values for $\Ind(\phi_1,\phi_2)$ are obtained by \cite[Chapter 5, Pg 64]{MR1280461}:
\begin{center}
\label{eq:rep_value_ind}
\begin{equation}
\Ind(\phi_1,\phi_2)((\sigma_1,\sigma_2)^\epsilon)=
\begin{cases}
\phi_1(\sigma_1)\phi_2(\sigma_2)+\phi_1(\sigma_2)\phi_2(\sigma_1) & \text{if
$\epsilon = 1$},\\
0 & \text{otherwise.}
\end{cases}
\end{equation}
\end{center} 

If however $\phi_1$ and $\phi_2$ are isomorphic, with $\phi$ the representative of their common isomorphism class, the induced representation $\Ind_{H_{k-1} \times H_{k-1}}^{H_k}(\phi \otimes \phi)$ is the sum of two irreducible representations of $H_k$. We call these the extensions of $\phi \otimes \phi$.  The restriction of either extension to $H_{k-1} \times H_{k-1}$ is $\phi \otimes \phi$. It remains to find the character values of the two extensions on classes of Type II (see Definition \ref{def:conj_type}). From \cite{MR409621} we have that the values of the two extensions of $\phi \otimes \phi$ on the class $(\Id,\sigma)^{-1}$ are $\phi(\sigma)$ and $-\phi(\sigma)$. Thus we denote these extensions $\Extp(\phi)$ and $\Extn(\phi)$ respectively. 

\begin{center}
\label{eq:rep_value_ass}
\begin{equation}
\Extpn(\phi)((\sigma_1,\sigma_2)^\epsilon)=
\begin{cases}
\phi(\sigma_1)\phi(\sigma_2) & \text{if
$\epsilon = 1$},\\
\pm \phi(\sigma_1 \sigma_2) & \text{otherwise.}
\end{cases}
\end{equation}
\end{center}

Now we define three types of representations, as we did for conjugacy classes in Definition \ref{def:conj_type}.
\begin{definition}
\label{def:rep_type} Given an irreducible representation $\phi$ of $H_k$, we have:
\begin{itemize}
\item $\phi$ is of Type I if $$\phi=\Extp(\phi_1),$$ for an irreducible representation $\phi_1$ of $H_{k-1}$.
\item $\phi$ is of Type II if $$\phi=\Extn(\phi_1),$$ for an irreducible representation $\phi_1$ of $H_{k-1}$.
\item $\phi$ is of Type III if $$\phi=\Ind(\phi_1,\phi_2),$$ for nonisomorphic irreducible representations $\phi_1$ and $\phi_2$ of $H_{k-1}$.
\end{itemize}
\end{definition}
These results are summarised in Table \ref{table:irreps}. Based on Table \ref{table:irreps} it may be observed that the character table of $H_k$ can be recursively obtained. The template for doing so is Table \ref{table:template}. The recursive process is illustrated for $k=2$ in Table~\ref{table:H2}.
 
\begin{table}[h]	
\begin{center}

\caption{Irreducible characters of $H_k$}
\label{table:irreps}
\scalebox{0.8}
{
\begin{tabular}{|c|c | c | c| c|}

\hline

Type&Notation&Description & Action on $(\sigma_1,\sigma_2)^1$ & Action on $(\Id,\sigma)^{-1}$\\

\hline
I&$\Extp(\phi)$ & Positive extension of  $\phi \otimes \phi$ & $\phi(\sigma_1)\phi(\sigma_2)$ & $\phi(\sigma)$ \\
\hline
II&$\Extn(\phi)$ & Negative extension of  $\phi \otimes \phi$ & $\phi(\sigma_1)\phi(\sigma_2)$ & $ -\phi(\sigma)$ \\
\hline  
III&$\Ind(\phi_1,\phi_2)$ & Induced from $\phi_1 \otimes \phi_2$ & $\phi_1(\sigma_1)\phi_2(\sigma_2)$ & 0\\
&&&$+ \phi_1(\sigma_2)\phi_2(\sigma_1)$ &\\
\hline
\end{tabular}}
\end{center}
\end{table} 

\begin{remark}
\label{rem:dim}
From Table \ref{table:irreps} we know the dimensions of the representations of each type. Thus we have $\Dim(\Extpn(\phi))= \Dim(\phi)^2$ and $\Dim(\Ind(\phi_1,\phi_2))= 2\Dim(\phi_1)\Dim(\phi_2)$. 
\end{remark}

\begin{table}[H]
\caption{Template for the character table for $H_k$}
\label{table:template}
\begin{center}
\scalebox{0.8}{
\begin{tabular}{|c |c | c | c|}

\hline

&Type I & Type II & Type III\\

$\Extp(\phi)$ & \cellcolor{green!25}$\phi(\sigma_1)\phi(\sigma_2)$ & \cellcolor{green!25}$\phi(\sigma)\phi(\sigma)$ & \cellcolor{blue!25}character table for $H_{k-1}$ \\

$\Extn(\phi)$ & \cellcolor{green!25}$\phi(\sigma_1)\phi(\sigma_2)$ & \cellcolor{green!25}$\phi(\sigma)\phi(\sigma)$ &\cellcolor{blue!50}-character table for $H_{k-1}$ \\

$\Ind(\phi_1,\phi_2)$ & \cellcolor{green!25}$\phi_1(\sigma_1)\phi_2(\sigma_2)+\phi_1(\sigma_2)\phi_2(\sigma_1) $ & \cellcolor{green!25}$2\phi_1(\sigma)\phi_2(\sigma)$ & 0\\
\hline  
\end{tabular}  
 }
\end{center}
\end{table}

\begin{table}[H]
\caption{Character table for $H_1$:}
\label{table:H1}

\begin{center}
\scalebox{1.0}{
\begin{tabular}{|c|c|c|}
\hline

& $C_1:= (\Id,\Id)^1$ & $C_2:= (\Id,\Id)^{-1}$\\

$\Extp(\Id)$ & \cellcolor{green!25}$1$ & \cellcolor{blue!25}$1$\\

$\Extn(\Id)$ &  \cellcolor{green!25}$1$ &  \cellcolor{blue!50}$-1$\\

 \hline 

 \end{tabular}
 }
\end{center}

\end{table}

\begin{table}[H]
\caption{Character table for $H_2$:}
\label{table:H2}

\begin{center}
\scalebox{0.8}{
\begin{tabular}{|c|c|c|c|c|c|}

\hline

&$(C_1,C_1)^1$ & $(C_2,C_2)^1$ & $(C_1,C_2)^1$ & $(\Id,C_1)^{-1}$ &
$(\Id,C_2)^{-1}$\\

$\Extp(\Extp(\Id)$& \cellcolor{green!25}$1$& \cellcolor{green!25}$1$&\cellcolor{green!25}$1$&\cellcolor{blue!25}$1$&\cellcolor{blue!25}$1$\\

$\Extp(\Extn(\Id)$& \cellcolor{green!25}$1$& \cellcolor{green!25}$1$&\cellcolor{green!25}$-1$&\cellcolor{blue!25}$1$&\cellcolor{blue!25}$-1$\\

$\Extn(\Extp(\Id)$& \cellcolor{green!25}$1$& \cellcolor{green!25}$1$&\cellcolor{green!25}$1$&\cellcolor{blue!50}$-1$&\cellcolor{blue!50}$-1$\\

$\Extn(\Extn(\Id)$& \cellcolor{green!25}$1$& \cellcolor{green!25}$1$&\cellcolor{green!25}$-1$&\cellcolor{blue!50}$-1$&\cellcolor{blue!50}$1$\\

$\Ind(\Extp(\Id),\Extn(\Id))$&\cellcolor{green!25}2&\cellcolor{green!25}-2&\cellcolor{green!25}0&0&0\\

\hline

 \end{tabular}
 }

\end{center}
\end{table}

\subsection{Binary trees and forests}      

Binary trees are commonly occuring objects in computer science and mathematics. For a complete introduction to these objects see \cite{MR3077152}. 

A rooted binary tree is  a tuple $(r,L,R)$- a root vertex $r$, and binary trees $L$ and $R$, denoted the left and right subtree. They are commonly depicted by connecting the root vertex $r$ to the root vertices of each of the subtrees $L$ and $R$. Given a vertex $y$ of a binary tree, it is known that there exists a unique path $r= v_0, v_1, \dotsc, v_k=y$. The height of the vertex $y$ is $k$- the number of vertices on this unique path (not counting the root vertex). Each vertex of a binary tree is connected to two possibly trivial subtrees. If both subtrees connected to a vertex are trivial, the vertex is called an external vertex. All vertices that are not external are called internal.     

For our purposes the designation of a subtree as either the right or the left is superfluous. Thus we may define binary trees formally as  a tuple $(r,S)$ of a root vertex $r$ and a possibly empty multiset $S$ of at most two binary trees. The height of a vertex is unaffected by this modification in definition. Binary trees where all the external vertices have the same height are called 1-2 binary trees. 

\begin{definition}
\label{def:tree}
A 1-2 binary tree of height $k$ is a tuple $(r,S)$ consisting of a root vertex $r$ and multiset $S$ comprising of upto two binary trees, where every external vertex of the tree has height $k$.   
\end{definition} 

We refer to 1-2 binary trees as either binary trees or trees when there is no ambiguity in doing so. A forest is a collection of trees. Given an integer $n=2^{k_1}+\dotsb+2^{k_s}$ with $k_1 > \dotsb> k_s$, and recalling convention that $Bin(n)=\{k_1,\dotsc,k_s\}$, we define:
\begin{definition}
\label{def:for}
A forest of size $n$ is an ordered collection of 1-2 binary trees $(T_1,\dotsc,T_s)$, where $T_i$ is a 1-2 binary tree of height $k_i$ for $i=1,\dotsc,s$.
\end{definition}

A forest with a single element is identified with the tree that is its only element.

\subsection{Irreducible representations and conjugacy classes as binary trees}

The number of irreducible representations of $H_k$, and the number of conjugacy classes of $H_k$, are both given by the recurrence:
\begin{align}
\label{eq:thenums}
a_k &=2a_{k-1} + {{a_{k-1}} \choose {2}},\\
a_0 &=1.   \nonumber
\end{align}

The sequence generated by Equation \eqref{eq:thenums} counts the number of 1-2 binary trees  of height $k$ (see \cite[Sequence A006893]{oeis}). This observation leads us to define bijections between 1-2 binary trees of height $k$  and  the set of irreducible representations of $H_k$ (also the set of conjugacy classes of $H_k$). 
\begin{definition}
\label{def:bij_rep}
Define a family of bijections $\theta_{2^k}$ for nonnegative integers $k$ between the set of irreducible representations of $H_k$ and the set of 1-2 binary trees of height $k$ as under:
\begin{equation}
\theta_{2^k}(\Gamma)=
\begin{cases}
(r,\{\theta_{2^{k-1}}(\phi),\theta_{2^{k-1}}(\phi)\}) &\Gamma=\Extp(\phi \otimes \phi),\\
(r,\{\theta_{2^{k-1}}(\phi)\}) &\Gamma=\Extn(\phi \otimes \phi),\\
(r,\{\theta_{2^{k-1}}(\phi_1),\theta_{2^{k-1}}(\phi_2)\}) &\Gamma=\Ind(\phi_1, \phi_2).
\end{cases}
\end{equation}
The dimension of a binary tree $T$ is denoted $\Dim(T)$ and is defined to be the dimension of its corresponding irreducible representation.
\end{definition}

\begin{definition}
\label{def:bij_con}
Choosing class representatives as in Table \ref{table:conj_class}, we define a bijection $\Theta_{2^k}$ between representatives of conjugacy classes of $H_k$ and 1-2 binary trees of height $k$ as under:
\begin{equation}
\label{eq:bij_con}
\Theta_{2^k}(\sigma)=
\begin{cases}
(r,\{\Theta_{2^{k-1}}(\sigma),\Theta_{2^{k-1}}(\sigma)\}) &\sigma=(\sigma,\sigma)^{1},\\
(r,\{\Theta_{2^{k-1}}(\sigma)\}) &\sigma=(\Id,\sigma)^{-1},\\
(r,\{\Theta_{2^{k-1}}(\sigma_1),\Theta_{2^{k-1}}(\sigma_2)\}) &\sigma=(\sigma_1,\sigma_2)^{1}.

\end{cases}
\end{equation}
The order of a binary tree $T$ is denoted $\Or(T)$ and is defined to be the size of its corresponding conjugacy class.
\end{definition}

Now we define three types of binary trees, in line with Definitions \ref{def:conj_type} and \ref{def:rep_type}.

\begin{definition}
\label{def:tree_type}
Given a 1-2 binary tree $T$ of height $k$, we have:
\begin{itemize}
\item $T$ is of Type I if $$T=(r,\{T_1,T_1\}),$$ for a 1-2 binary tree $T_1$ of height $k-1$.
\item $T$ is of Type II if $$T=(r,\{T_1\}),$$ for a 1-2 binary tree $T_1$ of height $k-1$.
\item $T$ is of Type III if $$T=(r,\{T_1,T_2\}),$$ for distinct 1-2 binary trees $T_1$ and $T_2$ of height $k-1$.
\end{itemize}
\end{definition}

These two families of bijections extend to the case where $n$ is an arbitrary integer as below:
\begin{definition}
\label{def:bij_n}
With $\theta_{2^k}$ as in Definition \ref{def:bij_rep} and $\Theta_{2^k}$ as in Definition \ref{def:bij_con} we define $\theta_n$ and $\Theta_n$ as follows:
\begin{align*}
\theta_n= \theta_{2^{k_1}} \times \dotsb \times \theta_{2^{k_s}}  \nonumber \\
\Theta_n= \Theta_{2^{k_1}} \times \dotsb \times \Theta_{2^{k_s}}.
\end{align*}
The \emph{dimension} of a forest and the \emph{order} of a forest are defined respectively to be the product of the dimensions and the product of the orders of the trees in the forest. 
\end{definition}
\subsection{Some properties of hook partitions of even size}
In what follows we present some results on the Littlewood-Richardson coefficients associated to a hook partitions $(a+1,1^b)$ (denoted $(a|b)$ in the Frobenius notation) of even size and the border-strip tableaux of such shapes. These results are required only in Section \ref{sec:res}, and are generally easy to prove and so proofs are either abridged or absent throughout this subsection.  

\begin{definition}
\label{def:half_par}
Given hook partition $\lambda= (a|b)$ of an even integer, define the partition $\frac{\lambda}{2}$ as:
\begin{equation*}
\frac{\lambda}{2}=\begin{cases}
(\frac{a-1}{2}|\frac{b}{2}) & {a \text{ odd},}\\
(\frac{a}{2}|\frac{b-1}{2}) & {\text{otherwise}.}
\end{cases}
\end{equation*}
\end{definition}

The Littlewood-Richardson coefficient $c_{\mu,\nu}^{\lambda}$ indexed by partitions $\mu,\nu,\lambda$ with $|\mu|+|\nu|=|\lambda|$ is defined to be the multiplicity of the Specht module $V_{\lambda}$ in $\Ind_{S_{|\mu|}\times S_{|\nu|}}^{S_{|\lambda|}} $. Note that $c_{\mu,\nu}^{\lambda}=c_{\nu,\mu}^{\lambda}$. 

The coefficient $c_{\mu,\nu}^{\lambda}$ may also be defined as the number of semistandard Young tableaux of skew shape $\lambda \setminus \mu$ and content $\nu$ whose reverse reading word is a lattice permutation. For a definition of these terms and a statement of this result see \cite[Theorem A1.3.3]{MR1676282}. 
 
\begin{prop}
\label{prop:lrcoeff}
Let $\lambda=(a|b)$ be a hook partition. For partitions $\mu$ and $\nu$ with $|\mu|+|\nu|=|\lambda|$, we have:
\begin{itemize}
\item $c_{\mu,\nu}^{\lambda}$ is either $0$ or $1$. If $c_{\mu,\nu}^{\lambda}>0$ then $\mu$ and $\nu$ are hook partitions.
\item With $\mu=(a_1|b_1) \subset \lambda$, $c_{\mu,\nu}^{\lambda}=1$ iff $\nu=(a-a_1-1|b-b_1)$ or $\nu=(a-a_1|b-b_1-1)$.
\item If $|\lambda|$ is even, then $c_{\mu,\mu}^{\lambda}=1$ iff $\mu=\frac{\lambda}{2}$.
\end{itemize}
\end{prop}

Thus restricting an irreducible representation of $S_n$ corresponding to a hook partition to a Young subgroup yields a sum of tensor products of Specht modules corresponding to hook partitions. The proof follows easily from the combinatorial definition of these coefficients.  

We are grateful to Steven Spallone for sharing his notes, which contain the preceding proposition as well as the following one and its proof.
 
\begin{prop}
\label{prop:lr_n}
Let $\lambda$ be a partition and $h$ be a hook in $\lambda$ with corresponding rim-hook $r$. Then
\begin{displaymath}
c_{\lambda/r,h}^{\lambda}=1.
\end{displaymath}
\end{prop}

\begin{proof}
We prove this by exhibiting a unique semistandard skew tableau of shape $r$ and content $h$ whose reverse reading word is a lattice permutation.
A cell of the rim-hook $r$ is said to be a lower cell if there exists another cell in $r$ directly above it, and is said to be an upper cell if not. Consider the labelling of $r$ by assigning a $1$ to all upper cells, and by numbering the lower cells $2,\dotsc, s$ from top to bottom. This process is illustrated in Figure \ref{fig:rise}.

The reverse reading word (entries of tableau read right to left and top to bottom) of this tableau is of the type $1\dots121\dots1\dots1 s1\dots1$, which is clearly a lattice permutation (the number of $i$s is more than the number of $i+1$s in any initial segment, for all positive integers $i$).

To show that it is unique, assume that it is so for any initial segment of the rim $r$ of length less than $l$ beginning at the north-east end. Let $k$ be the largest value assigned to a lower cell in this numbering. Then the value of the next cell must be either $1$ or $k+1$, by consideration of the content. It is easy to see that if this cell is an upper cell it must be labelled $1$, and if a lower cell must be labelled $k+1$, if the reverse reading word is to be a lattice permutation. Thus the labelling is uniquely specified. 
\end{proof}

\begin{figure}
\[
\begin{ytableau}
       \none&\none&\tikznode{a1}{1}  & \tikznode{a2}{1} &\tikznode{a3}{1}\\
       \none&\none &\tikznode{a4}{2} \\
       \none&\none &\tikznode{a5}{3}\\
        \tikznode{a6}{1} & \tikznode{a7}{1} &\tikznode{a8}{4}\\
        \tikznode{a9}{5}
\end{ytableau}
\]
\caption{The unique tableau of shape $(5,3^3,1)\setminus (2^3)$ and content $(4|4)$ whose reverse reading word is a lattice permutation (namely, $111234115$).}
\label{fig:rise}
\end{figure}

Now we turn to border-strip tableaux of hook shapes. A border strip is a connected skew shape such that if $(i,j)$ belongs to the shape then $(i+1,j+1)$ does not. A border-strip tableau of shape $\lambda$ is defined (see \cite{MR1676282}) to be a sequence $\lambda_0=\emptyset \subset \lambda_1 \subset \dots \subset \lambda_l=\lambda$, where $\lambda_i \setminus \lambda_{i-1}$ is a border strip, for all $i=1,\dotsc,l$. We denote the border strip $\lambda_i \setminus \lambda_{i-1}$ on the shape $\lambda$ by populating the cells of this border strip with the integer $i$. 
 
\begin{definition}
\label{def:istrip}
For a positive integer i, the $i$-strip of a border-strip tableau $T$ is the border strip in $T$ that is filled with the integer $i$.
\end{definition}
The size of the $i$-strip is the number of cells in the border strip. The height of the $i$-strip is denoted $ht(i)$ and is one less than the number of rows the $i$-strip occupies.

\begin{definition}  
The content of a border-strip tableau $T$ is the vector $(a_1,a_2,\dotsc)$ of nonnegative integers where $a_i$ is the size of the $i$-strip, for $i \geq 1$.
\end{definition}

The set of border-strip tableaux of shape $\lambda$ and content $\nu$ is denoted $\text{BST}(\lambda,\nu)$. 
  
\begin{definition}
The height of a border strip tableau $T$ is defined as:
$$ht(T)= \sum_{i\geq 1}ht(i).$$
\end{definition}
 
Here we restrict ourselves to border strip tableaux of hook partitions of even size, where the size of each $i$-strip is an even integer. There is a unique tiling of such a hook shape with dominoes, and each $i$-strip can be regarded as a union of adjoining dominoes, where both cells of each domino are filled with the integer $i$. The tableau on the left of Figure \ref{fig:domino_tab} is an example of such a shape.

Now given a hook shape $\lambda$ and a vector $\nu$ comprising only even integers, let $\frac \nu 2$ denote the vector whose entries are half the corresponding entries in $\nu$. Consider the map: 
\begin{displaymath}
\psi: \text{BST}(\lambda,\nu)\rightarrow \text{BST}(\frac\lambda 2,\frac \nu 2),
\end{displaymath}
which replaces each domino in a tableau of shape $\lambda$ by a cell with the same content. Figure \ref{fig:domino_tab} provides an example of $\psi$. 
\begin{figure}
\ytableausetup{centertableaux}
\[
\begin{ytableau}
       \tikznode{a1}{1} & \tikznode{a2}{2}&\tikznode{a3}{2} &\tikznode{a4}{2} & \tikznode{a5}{2}\\       
       \tikznode{a6}{1}\\
       \tikznode{a7}{3}\\
       \tikznode{a8}{3}\\
\end{ytableau}
~\xrightarrow{\psi}~
\begin{ytableau}
       \tikznode{a1}{1} & \tikznode{a2}{2}&\tikznode{a3}{2}\\       
       \tikznode{a4}{3}\\
\end{ytableau}
\]
\caption{A tableau in $\text{BST}((4|3),(2,4,2))$ and its image under the map $\psi$.}
\label{fig:domino_tab}
\end{figure}
\begin{prop}
\label{prop:half}
Let $\lambda=(a|b)$ be a hook partition of even size, $\nu$ be a vector of even integers and $P \in \text{BST}(\lambda,\nu)$. Then we have:
\begin{itemize}
\item The map $\psi$ is surjective.
\item Two cells in $\psi(P)$ are connected vertically or horizontally if the corresponding dominoes in $P$ are connected either vertically or horizontally respectively.
\item Let $\text{vert}(\lambda)$ denote the number of vertical dominoes of $\lambda$. Then:
$$ht(P)=ht(\psi(P))+\text{vert}(\lambda).$$
 
\end{itemize}
\end{prop} 
\begin{proof}
The first of these assertions is easy to prove. In particular a tableau of shape $(a|b)$ is the image under $\psi$ of two tableau- one each of shape $(2a|2b+1)$ and $(2a+1|2b)$.

The second assertion is trivial to prove as well.  We give a sketch of the argument for the third.
 
The height of an $i$-strip in $P$ is either:  
$$ht(i)= 2|\text{vert}_i(P)|-1,$$ if either the strip does not contain the cell $(1,1)$ or if it does and the cell $(1,1)$ belongs to a vertical domino. If the strip contains the cell $(1,1)$ and this cell belongs to a horizontal domino:
$$ht(i)= 2|\text{vert}_i(P)|.$$ 
  
Since each domino is changed to a cell by $\psi$, the height of the $i$-strip in $\psi(P)$ is either:
$$ht(i)= |\text{vert}_i(P)|-1,$$ if either the $i$-strip in $P$ does not contain the cell $(1,1)$ or if it does and the cell $(1,1)$ belongs to a vertical domino, or
$$ht(i)= |\text{vert}_i(P)|,$$ if the $i$-strip in $P$ contains the cell $(1,1)$ and this cell belongs to a horizontal domino. 
The difference between these two expressions is precisely $\text{vert}_i(P)$. Summing over all $i$-strips proves the result.
\end{proof}

\begin{remark}
Proposition \ref{prop:half} is a special case of the bijection in \cite[Theorem 25]{MR1855862} between domino tableaux of shape $\lambda$ and pairs of tableaux $(T_1,T_2)$ where the shapes of $T_1$ and $T_2$ are the two quotient of $\lambda$. When $\lambda$ is a hook shape, one of the two shapes in the quotient is empty, motivating the map defined above. Note, however, that the map $\psi$ is not a bijection.    
\end{remark}
 
 \section{The Bratteli diagram}
 \label{sec:brat}
 Given the family of subgroups $\{P_n\}_{n \geq 0}$ with $P_0 \subset \dotsb P_{n-1} \subset P_{n} \subset \dotsb$, the Bratteli diagram of this family, denoted  $\mathbb{P}$, is the graded poset whose vertices at the $n$th level are indexed by the irreducible representations of $P_n$, for all $n \geq 0$.  An edge exists between the vertex corresponding to the representation $\gamma$ of $P_{n-1}$ and the vertex corresponding to the representation $\Gamma$ of $P_{n}$ if $\gamma$ is a constituent of $\Res_{P_{n-1}}^{P_n}(\Gamma)$. 
 
Observe that  $P_{2^k-1}=H_{k-1} \times P_{2^{k-1}-1}$. Thus we have:
$$\Res^{H_k}_{P_{2^k-1}}(\Gamma)= \Id \times \Res^{H_{k-1}}_{P_{2^{k-1}-1}}\circ  \Res^{H_k}_{H_{k-1} \times H_{k-1}}(\Gamma).$$ We denote this restriction by $r_{2^k}$, and note that:
 
\begin{equation}
\label{eq:rk}
r_{2^k}=(\Id \times r_{2^{k-1}}) \circ \Res^{H_k}_{H_{k-1} \times H_{k-1}}.
\end{equation}

For instance if $\Gamma$ is an irreducible representation of $H_k$ of Type III (see Table \ref{table:irreps}) and $\Res^{H_k}_{H_{k-1} \times H_{k-1}}(\Gamma)=\phi_1 \otimes \phi_2 + \phi_2 \otimes \phi_1$,  then:
\begin{equation*}
r_{2^k}(\Gamma)=\phi_1 \otimes r_{2^{k-1}}(\phi_2)+\phi_2 \otimes r_{2^{k-1}}(\phi_1).
\end{equation*}

Following as before the convention $Bin(n)=\{k_1,\dotsc,k_s\}$, an irreducible representation of $P_n$ is of the form $\phi_1 \otimes \dotsb \phi_s$, where $\phi_i$ is an irreducible representation of $H_{k_i}$ for $i=1,\dotsc,s$. We extend Equation \eqref{eq:rk} to the general case by restricting the last component, $\phi_s$. Thus:

\begin{equation}
\label{eq:rn}
r_{n}=\Id \times \dotsb \Id \times r_{2^{k_s}}.
\end{equation}

\begin{definition}
\label{def:down_rep}
Let $\Gamma$  be an irreducible representation of $P_n$, for some $n \geq 1$. Then the down-set of $\Gamma$, denoted $\Gamma^{-}$, is defined to be the multiset of representations of $P_{n-1}$ with each representation $\gamma$ occuring in $\Gamma^-$ as many times as $\gamma$ occurs in $r_{n}(\Gamma)$. The up-set of $\Gamma$, denoted denoted $\Gamma^{+}$, is defined to be the multiset of representations of $P_{n+1}$ such that $\Gamma$ occurs in their down-set, each repeated as many times as $\Gamma$ occurs in its down-set.
\end{definition}

We make the analogous definition of the down-set and up-set of a 1-2 binary tree:
\begin{definition}
\label{def:down_tree}
Given a forest $F$ of size $n$, let $F=\theta_{n}(\Gamma)$  as in Definition \ref{def:bij_rep} for some irreducible representation $\Gamma$ of $P_n$. Define the down-set $F^{-}$ and the up-set $F^{+}$ as the multisets obtained as under:  
\begin{align*}
F^{-}=\{\theta_{n-1}(\gamma)| \gamma \in \Gamma^{-}\},\\
F^{+}=\{\theta_{n+1}(\gamma)| \gamma \in \Gamma^{+}\}.
\end{align*}
\end{definition}

We define an operation on 1-2 binary trees to retrieve the downset $\tau^-$ from the tree $\tau$:
\begin{definition}
\label{def:res_tree}
\begin{align*}
\Res(\tau) = \begin{cases} 
T \times \Res(T) &\tau=(r,\{T,T\}),\\
T \times \Res(T) &\tau=(r,\{T\}),\\
T_1 \times \Res(T_2) \cup T_2 \times \Res(T_1) & \tau=(r,\{T_1,T_2\}), \\
\emptyset &\tau=(r,\emptyset).
\end{cases}
\end{align*}
\end{definition}
\begin{remark}
For the trees $\tau=(r,\{T,T\})$ and $\tau=(r,\{T\})$, $\Res(\tau)$ is the multiset formed by adding $T$ to the front of every forest in $\Res(T)$. If $\tau=(r,\{T_1,T_2\})$, and $T_1$ and $T_2$ are distinct subtrees, the multiset $\Res(\tau)$ is a union of two multisets, one where $T_1$ was added to the front of every forest of $\Res(T_2)$, and one where $T_2$ was added to the front of every forest of $\Res(T_1)$. The two terms of this union are disjoint since the largest trees in the forests of each ($T_1$ and $T_2$) are distinct. 
\end{remark}
\begin{prop}
\label{prop:res_tree}
Let $T$ be a 1-2 binary tree, and $T^{-}$ be its downset. Then $T^{-}=\Res(T)$.
\end{prop}

\begin{proof}
The proof follows by induction on the height of the tree. When $k=0$, let $\Id=(r,\emptyset)$ denote the trivial representation of $H_0$. Then the result is a matter of definition, so we begin with $k=1$. There are two trees of height $1$, namely $T_1= (r,\{\Id,\Id\})$ and $T_2=(r,\{\Id\})$. These correspond to the two representations in Table ~\ref{table:H1}. As can be seen, the restriction of both of these representations to the diagonal subgroup $D \cong H_0$ in $H_0 \times H_0$ is the trivial representation. The result of $\Res(T_1)$ and $\Res(T_2)$ is  $\Id$.

Now assume the result is true for trees of height less than $k$. Let $\Gamma$ be the representation of $H_k$ corresponding to $T$.  
If $T$(and thus $\Gamma$) are of type III(see Table \ref{table:irreps}), and let $\Gamma=Ind(\phi_1, \phi_2)$, then by Equation ~\eqref{eq:rep_value_ind} and Equation ~\eqref{eq:rk}:
\begin{align*}
\Res^{H_k}_{P_{2^k-1}}(\Gamma)=\sum_{i,j \in {1,2}, i \neq j} \phi_i \otimes \Res^{H_{k-1}}_{P_{2^{k-1}-1}}(\phi_j).  
\end{align*} 
Since the result holds for $k-1$, and by Definition \ref{def:bij_rep}, $T_i$ is the tree corresponding to $\phi_i$ for $i=1,2$, the down-set $T^{-}$ is $\Res(T)$.
A similar computation for trees of type I and type II completes the proof.
\end{proof}

From Definition \ref{def:res_tree} and Proposition \ref{prop:res_tree} we have:

\begin{corollary}
\label{cor:downset_k}
Given a tree $\tau$ of height $k$:
\begin{align*}
\tau^{-}=
\begin{cases}
T \times T^{-} &\tau=(r,\{T,T\}),\\
T \times T{-} &\tau=(r,\{T\}),\\
T_1 \times T_2^{-} \cup T_2 \times T_1^{-} & \tau=(r,\{T_1,T_2\}), \\
\emptyset &\tau=(r,\emptyset).
\end{cases}
\end{align*}
\end{corollary}

\begin{corollary}
\label{cor:upset_k}
Given a forest $F$ of size $2^k-1$, let $F(1)$ denote the largest tree in the forest and $\overline{F}$ denote the tuple $F$ with the tree $F(1)$ removed:
\begin{align*}
F^{+}=
\begin{cases}
\{(r,\{F(1),T\})|T \in \overline{F}^{+}\} &F(1) \not \in \overline{F}^{+},\\
\{(r,\{F(1),T\})|T \in \overline{F}^{+}\} \cup \{(r, \{F(1)\})\} &F(1) \in \overline{F}^{+}.
\end{cases}
\end{align*}
\end{corollary}

\begin{corollary}
\label{cor:mult_k}
Given a tree $T$ of height $k$ for any integer $k \geq 1$, let $\text{set}(T^{-})$ denote the set of distinct elements in $T^{-}$. Then $T^{-}=\text{set}(T^{-})$.
\end{corollary}

\begin{proof}
The proof proceeds by induction. The result is easily verified for $k=1$. If it holds for all integers less than $k$, consider the down-set of a tree $T$ of height $k$:

If $T$ is of Type I or Type II, by corollary \ref{cor:downset_k}, the down-set of $T$ may be identified with the down-set of its only distinct subtree (by deleting the largest tree from each forest). Then $T^{-}=\text{set}(T^{-})$ by the induction hypothesis.

If $T$ is of Type III, and let $T_1$ and $T_2$ denote its distinct subtrees. Again by \ref{cor:downset_k} we know that $T^{-}$ is the union of two sets (multisets that are known to be multiplicity free by the induction hypothesis). One of these sets consists of forests where the largest tree is $T_1$, while the other consists of forests where the largest tree is $T_2$. This union is disjoint since these trees are distinct. 
\end{proof}

The operator $\Res$ can be extended to forests of arbitrary size, following the cue of Equation \eqref{eq:rn}. With $F=(T_1,\dotsc,T_n)$ a forest of size $n$, we have: 

\begin{definition}
\label{def:res_forest}
Define an operator $\Res$ from forests to multisets of forests as under:
\begin{displaymath}
\Res(F)= (T_1,\dotsc,T_{s-1})\times \Res(T_s). 
\end{displaymath}
\end{definition}

The following proposition is easy to observe from this defintion. 

\begin{prop}
\label{prop:res_forest}
Let $F$ be a forest of binary trees, and $F^{-}$ be its down-set. Then $F^{-}=\Res(F)$.
\end{prop}

We may combine these results into a combinatorial branching rule on forests of binary trees. Recall that a tree $T$ of height $k$ is identified with the forest $F=(T)$ of size $2^k$.

\begin{theorem}
\label{th:mainA}
Given a forest $F=(F_1,\dotsc,F_s)$ of size $n$:
\begin{enumerate}
\item
Define  $\underline{F}$ to be the forest $F$ without the element $F_{s}$.
The down-set of $F$ is given by:
\begin{align*}
%\label{eq:down_gen}
F^{-}= \underline{F} \times F_{s}^{-}
\end{align*}\\
\item
Let $d$ denote the smallest nonnegative integer that does not occur in $Bin(n)$. Partition $F$ as the tuple $F_1\times F_2$, where $F_1$ is the tuple of trees in $F$ with height greater than $d$, and $F_2$ is the tuple of trees of height less than $d$. Then the up-set of $F$ is given by:
\begin{align*}
%\label{eq:up_gen}
F^{+}=F_1 \times F_2^{+}.
\end{align*}
\end{enumerate}
\end{theorem}

Thus the branching at each level replicates the branching at the $2^k$th level, for some nonnegative integer $k$.

\begin{prop}
\label{cor:mult}
The branching in $\mathbb{P}$ is multiplicity free.
\end{prop}
\begin{proof}
We may reduce the proof to the branching at every $2^k$th level, and we know by Corollary \ref{cor:mult_k} that this is multiplicity free. From Theorem \ref{th:mainA}, we have this in general. 
\end{proof}

Another consequence of Theorem \ref{th:mainA} is the self-similarity of $\mathbb{P}$:
\begin{lemma}
\label{lem:branch}
Given a forest $F=(F_1,\dotsc,F_s)$ of size $2^k+m$ ($0 \leq m < 2^{k}$), let $\overline{F}=(F_2,\dotsc,F_s)$. For $m <n \leq 2^k$, let $\mathbb{P}_{F}^{n}$ denote the subgraph of $\mathbb{P}$ comprising $F$ and all forests of size at least $2^k+m$ and strictly less than $2^k+n$ that are comparable to $F$. Then:
$$\mathbb{P}_{F}^{2^{k}}=\{F_1\} \times \mathbb{P}_{\overline{F}}^{2^{k-1}}.$$
\end{lemma}

We end the section with Figure \ref{fig:bratteli} which shows a portion of the Bratteli diagram $\mathbb{P}$. 

\begin{figure}[htp]

\centering

%\scalebox{0.8}
\includegraphics[width=\textwidth,center]{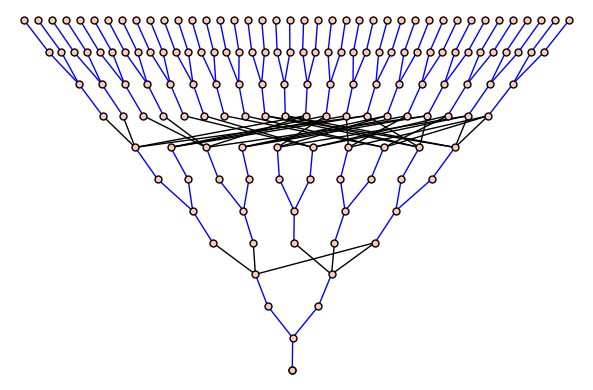}

\caption{Branching of irreducible representations for $n\leq11$.}

\label{fig:bratteli}

\end{figure}

\section{The one-dimensional representations of $P_n$}
\label{sec:one_dim}

We now turn to the subposet of one-dimensional representations of $\mathbb{P}$. Theorem 1 of \cite{MR3510808} states that the subgraph of odd partitions in Young's lattice is a binary tree that branches at every even level. We see that the subposet of one-dimensional representations of the family $\{P_n\}$ also has the structure of a binary tree (see Figure \ref{fig:rec_one}). We evince that these graphs are nonisomorphic by describing the structure of the subgraph of one-dimensional representations of $\mathbb{P}$, which we contrast with the description of the Macdonald tree in \cite{MR3510808}.

By Remark \ref{rem:dim} we conclude that an irreducible representation $\phi$ of $H_k$ is one-dimensional if $\phi=\Extpn(\phi_1)$ for an irreducible one-dimensional representation $\phi_1$ of $H_{k-1}$. 

\begin{definition}
\label{def:beta_k}
Define recursively a binary encoding of one-dimensional trees, $\beta_{2^k}$ acting on one-dimensional trees of height $k$ as below:
\begin{displaymath}
\beta_{2^k}(\tau)=
\begin{cases}
0\beta_{2^{k-1}}(T) & \tau=(r,\{T,T\}),\\
1\beta_{2^{k-1}}(T) & \tau=(r,\{T\}).
\end{cases}
\end{displaymath}
\end{definition}

For instance if for the tree $T$, $\beta_{2^{k-1}}(T)=b_1 b_1 \dotsc b_s$, then $\beta_{2^k}((r,\{T,T\}))=0b_1 b_1 \dotsc b_s$ and $\beta_{2^k}((r,\{T\}))=1b_1 b_1 \dotsc b_s$.

Thus we have an encoding of one-dimensional binary trees as binary strings. The family of maps $\beta_{2^k}$ may be extended to $\beta_n$, acting on every tree in a forest of size $n$. Thus, with $Bin(n)=\{k_1,\dotsc,k_s\}$:
\begin{equation}
\label{eq:beta_n}
\beta_n=\beta_{k_1} \times \dotsb \beta_{k_s}.
\end{equation}

\begin{definition}
\label{def:seqstring}
A sequence of strings of size $n$ is an ordered collection of binary strings $(b_1,\dotsc,b_s)$ where the length of the string $b_i$ is $k_i$ for $i=1,\dotsc,s$. 
\end{definition}

We now define an operation $\Res$ on binary strings, that is analogous to the operation of the same name defined on binary trees in Definition \ref{def:res_tree}:

\begin{definition}
\label{def:res_bin_string}
Given a binary string $b$ of length $k$, let $\overline{b}$ be the binary string of length $k-1$ obtained by removing the leading bit of $b$. Then
\begin{align*}
\Res(b) &= \overline{b} \times \Res(\overline{b}),\\ \nonumber
\Res(0) &=\emptyset,\\ \nonumber
\Res(1) &=\emptyset.
\end{align*}
\end{definition}

\begin{remark}
\label{rem:bin-res}
Observe that $\Res(b)= \{(\overline{b},\overline{\overline{b}},\dotsc)\}$. For instance $\Res(010)=\{10,0\}$. 
\end{remark}

\begin{lemma}
If $T$ is a one-dimensional tree of height $k$: $$\Res(\beta_{2^k}(T))=\beta_{2^k-1}(\Res(T)).$$
\end{lemma}
\begin{proof}
This is a straightforward proof by induction. Recall from Definition \ref{def:res_tree} that $\Res(T)= T \times \Res(T_1)$, where $T_1$ is the single unique subtree of the one-dimensional tree $T$. For $k=1$, the lemma is true by definition.

Assume it is true for all trees of height less than $k$. It is true also for $k$ if $\beta_{2^{k-1}}(T_1)=\overline{\beta_{2^k}(T)}$. This is so by the construction of $\beta$. 
\end{proof}  

This verifies that the operation $\Res$ defined on binary strings returns the down-set of the corresponding one-dimensional binary tree. We may extend this operation to act on sequences of binary strings in a manner analogous to Equation \ref{def:res_forest}. Given a sequence of strings $S=(b_1,\dotsc,b_s)$ of size $n$:

\begin{equation}
\label{eq:res_bin_sequence}
\Res(S) = (b_1,\dotsc,b_{s-1}) \times \Res(b_s).
\end{equation}

\begin{corollary}
\label{cor:cor_dim}
If $F$ is a one-dimensional forest of size n: $$\Res(\beta_{n}(F))=\beta_{n}(\Res(F)).$$  
\end{corollary}

The result of Corollary \ref{cor:cor_dim} is that we may identify subposet of one-dimensional representations of $\mathbb{P}$ with a poset generated by sequences of binary strings with $\Res$ providing the partial order. We denote by $\mathbb{B}$ the set of all sequences of strings of all positive integers. 

\begin{theorem}
\label{th:main}
The subgraph of one-dimensional irreducible representations for $\{P_n\}_{0 \leq n \leq s}$ is isomorphic to $(\mathbb{B},\Res)$.
\end{theorem} 
\begin{proof}
From Equation \eqref{eq:beta_n} there is a bijection between one-dimensional representations of $P_n$ and sequences of binary strings of size $n$. The down-set of a one-dimensional forest is a singleton set. From Corollary \ref{cor:cor_dim} we see that the operation $\Res$ acting on sequences of strings corresponding to a forest $F$ returns the binary encoding under Equation \eqref{eq:beta_n} of the unique element in $F^{-}$.
\end{proof}

Given a binary string $S$, let $F$ denote the forest it corresponds to. Then we define the down-set $S^{-}$ and the up-set $S^{+}$ to be $F^{-}$ and $F^{+}$ respectively. Note that $S^{-}$ is a singleton set. 

\begin{theorem}
\label{th:mainB}
Given an integer $n$ and a sequence of strings $S$ of size $n$ corresponding to a forest $F$, define $S(1)$ to be the longest string in $S$, and define $\overline{S}$ to be the sequence $S$ without $S(1)$. Similarly define $S_{min}$ to be the smallest string in $S$ and $\underline{S}$ to be the sequence $S$ without $S_{min}$. 
\begin{enumerate}
\item
The down-set of $S$ is given by:
\begin{align*}
%\label{eq:down_gen}
S^{-}= \underline{S} \times S_{min}^{-}
\end{align*}

\item
Partition $S$ as the tuple $S_1\times S_2$, where $S_1$ is the tuple of strings in $S$ with more than $d$ bits, and $S_2$ is the tuple of strings with less than $d$ bits. 

The up-set of $S$ is given by:
\begin{align*}
%\label{eq:up_gen}
S^{+}=
\begin{cases}
\{S_1 \times 0S_2(1), S_1 \times 1S_2(1)\} & S_2(1) \in \overline{S_2}^{+},\\
\emptyset & \text{ otherwise}
\end{cases}
\end{align*}
\end{enumerate}
\end{theorem}
\begin{remark}
$(\mathbb{B},\Res)$ (hereafter referred to as $\mathbb{B}$ when there is no ambiguity) is a binary tree that branches at every even level. Let $\mathbb{B}_k$ denote the first $2^k-1$ levels of $\mathbb{B}$. The following procedure constructs $\mathbb{B}_k$ recursively:
\begin{enumerate}
\item For each binary string $b$ of length $k-1$, let $v_b= (b,\overline{b},\overline{\overline{b}},\dotsb,\emptyset)$.
\item To each vertex $v_b$ of $\mathbb{B}_{k-1}$, attach two copies of $\mathbb{B}_{k-1}$, and denote them the left and right subtree of $v_b$.
\item Change the label of each vertex $v$ of the left subtree by appending the string $0b$ to the sequence. Similarly append $1b$ to the string labelling each vertex on the right subtree. 
\end{enumerate}. 
Figure \ref{fig:rec_one} uses this method to build the structure $\mathbb{B}_{3}$ from $\mathbb{B}_{2}$.
\end{remark}

\begin{figure}
    \centering
    \begin{minipage}[b]{0.4\textwidth}
        \centering
        \includegraphics[width=0.5\textwidth]{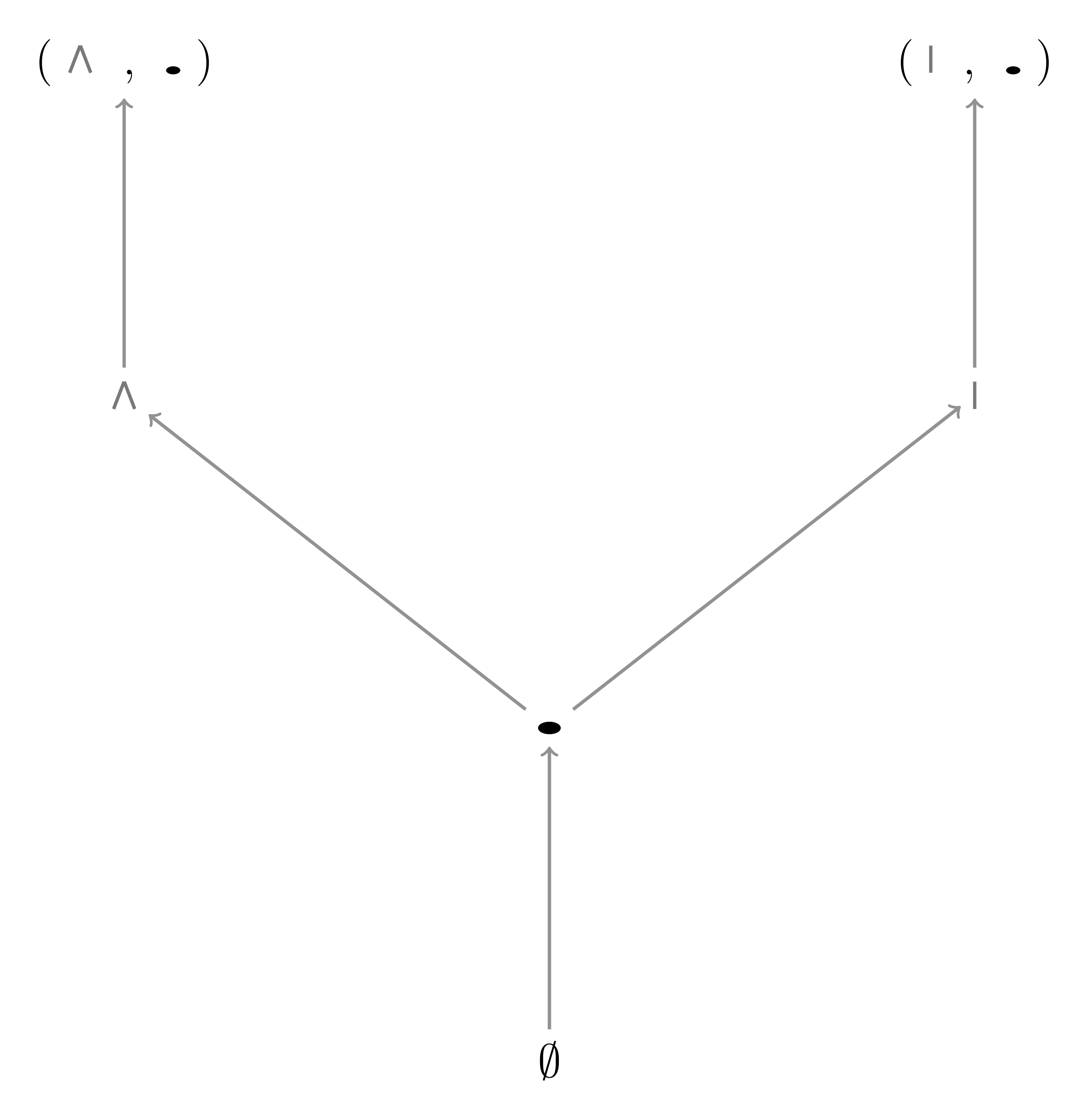} % first figure itself
        \caption*{$\mathbb{B}_2$}
    \end{minipage}\hfill
    \begin{minipage}[b]{0.9\textwidth}
        \centering
        \includegraphics[width=\textwidth]{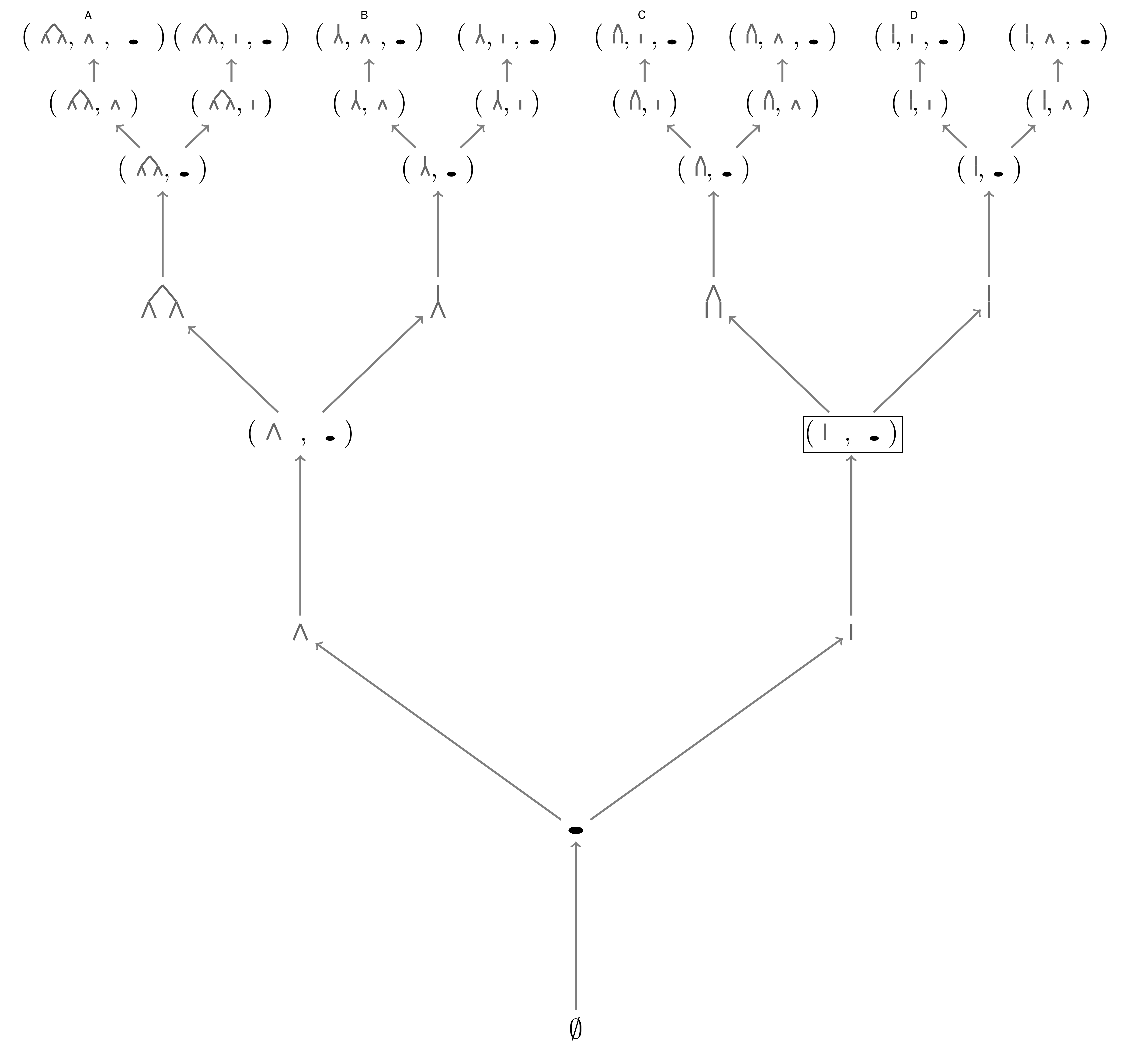} % second figure itself
        \caption*{$\mathbb{B}_3$}
    \end{minipage}
\caption{$\mathbb{B}_3$ is built recursively by attaching two copies of $\mathbb{B}_2$ to appropriate nodes on the maximal level of $\mathbb{B}_2$. Nodes on the highest level of $\mathbb{B}_3$ that further propagate are labelled A-D.}
\label{fig:rec_one}
\end{figure}
A recursive construction of the Macdonald tree can be found in \cite{MR3510808}. In particular the Macdonald tree has only two infinite rays. The subgraph  $\mathbb{B}$ by contrast has an infinite number of infinite rays, since each binary string $b$ can be extended by attaching $\epsilon=0,1$ to the left of $b$, and between the vertices $\epsilon b$ and $b$, there is a unique path in $\mathbb{B}$.

\begin{corollary}
\label{cor:main}
The Macdonald tree is not isomorphic to $\mathbb{B}$.
\end{corollary}

These differences may be observed by contrasting Figure \ref{fig:amritree} to Figure \ref{fig:bratelli}.

\begin{figure}[htp]
\centering
\includegraphics[width=\textwidth,center]{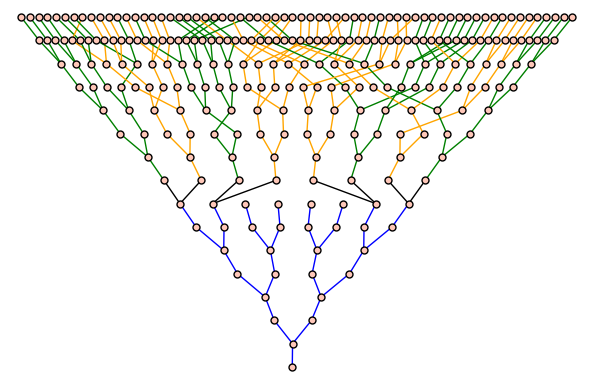}
\caption{The Macdonald tree for levels $n \leq15$}
\label{fig:amritree}
\end{figure}

\begin{figure}[htp]
\centering
\includegraphics[width=\textwidth,center]{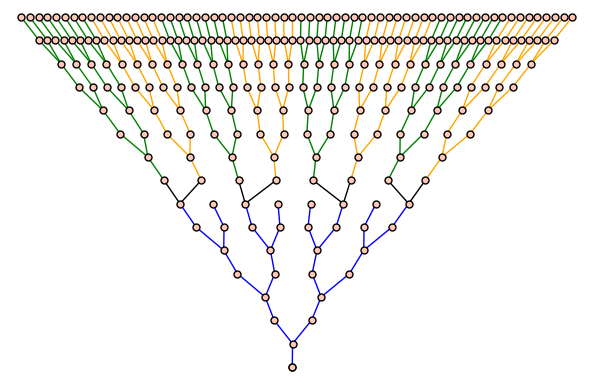}
\caption{The subgraph of one-dimensional representations of $P_n$ for levels $n \leq15$.}
\label{fig:bratelli}
\end{figure}

\section{Restrictions of odd-dimensional representations}
\label{sec:res}
In this section we consider the irreducible representations that occur in the restriction of an odd-dimensional representation of a symmetric group to a Sylow 2-subgroup. We  provide a recursive formula for the multiplicities of every irreducible representation of the subgroup when $n=2^k$. In general we have only a sufficient condition for a representation to occur. An interesting bijection between odd partitions of $n$ and one-dimensional representations of a Sylow subgroup of $S_n$ was found by Giannelli et al. in \cite{MR3687936}. Here it was shown that for $n=2^k$, a unique one-dimensional representation occurs in the restriction of a representation corresponding to hook partitions. In this case the bijection maps the hook partition to this representation. Using a result in \cite{MR3510808} on a unique decomposition of odd-partitions into hooks, they extended this bijection to all positive integers. We give a recursive description of this unique one-dimensional representation when $n=2^k$. Throughout this section, binary trees of Type I, II and III are denoted by $(r,\{T_1,T_1\})$, $(r,\{T_1\})$ and $(r,\{T_1,T_2\})$ respectively. For trees of Type I and Type II, define $\text{ext}(T)=1$ if $T$ is of Type I and $\text{ext}(T)=-1$ if T is of Type II. 

\begin{definition}
\label{def:moment}
Given two class functions $\chi_1$ and $\chi_2$ of a group $G$
\begin{displaymath}
\langle \chi_1, \chi_2 \rangle_{G} = \frac{1}{|G|}\sum_{g \in G} \chi_1(g^{-1})\chi_2(g).
\end{displaymath}
\end{definition}

Let $\chi_\lambda$ denote the irreducible character of the symmetric group $S_{n}$ corresponding to the partition $\lambda$ of n, and let $\chi_F$ denote the irreducible representation of $P_n$ corresponding to the forest $F$ of size $n$. We know that $\langle \chi_\lambda, \chi_F \rangle_{P_n}$ is the multiplicity of $\chi_F$ in the restriction of $\chi_\lambda$ to $P_n$. 

\begin{theorem}
\label{th:mainC}
Given a partition $\lambda$ of $2^k$ and a tree $T$ of height $k$, let $r_{\lambda T}=\langle \chi_\lambda, \chi_T \rangle_{H_k}$. Then we have:
\begin{align*}
r_{\lambda T} = 
\begin{cases}
A(T_1,\lambda)+\frac{1}{2}r_{\frac{\lambda}{2}T_1}^2+\frac{\text{ext}(T)(-1)^{\text{vert}(\lambda)}}{2}r_{\frac{\lambda}{2}T_1} & \text{T Type I or II},\\
B(T_1,T_2,\lambda)+r_{\frac{\lambda}{2}T_1}r_{\frac{\lambda}{2}T_2} & \text{T Type III.}
\end{cases} 
\end{align*}
The quantities A and B are defined as under
\begin{align}
\label{eq:A}
A(T_1,\lambda)&=\sum_{(a_1|b_1)}r_{(a_1|b_1)T_1}(
r_{(a-a_1|b-b_1-1)T_1}+r_{(a-a_1-1|b-b_1)T_1}),
\end{align}
where the sum is over hook partitions $(a_1|b_1)$ of $2^{k-1}$ with $\frac{\lambda}{2} \neq (a_1|b_1)$ and $ a \geq a_1 >\frac{a}{2}$, and
\begin{align}
\label{eq:B}
B(T_1,T_2,\lambda)&=\sum_{(a_1|b_1)}r_{(a_1|b_1)T_1}(r_{(a-a_1|b-b_1-1)T_2}+r_{(a-a_1-1|b-b_1)T_2})+ \\
\nonumber
&r_{(a_1|b_1)T_2}(r_{(a-a_1|b-b_1-1)T_1}+r_{(a-a_1-1|b-b_1)T_1}),
\end{align}
where the sum is over the same range as \eqref{eq:A}.
\end{theorem}
\begin{proof}
We write out the expression for $r_{\lambda T}$ as in Definition \ref{def:moment}
$$r_{\lambda T}= \frac{1}{|H_k|}\sum_{g \in H_k} \chi_\lambda(g)\chi_T(g),$$
noting that $\chi_\lambda(g^{-1})=\chi_\lambda(g)$.

The set of elements of $H_{k}$ may be split into two sets: $\{(\sigma_1,\sigma_2)^{1}|\sigma_1,\sigma_2 \in H_{k-1}\}$ and $\{(\sigma_1,\sigma_2)^{-1}|\sigma_1,\sigma_2 \in H_{k-1}\}$. Splitting the sum thus
\begin{align}
\label{eq:moment}
r_{\lambda T}=\frac{1}{|H_k|}\sum_{g=(\sigma_1,\sigma_2)^1}\chi_\lambda(g)\chi_T(g)+\frac{1}{|H_k|}\sum_{g=(\sigma_1,\sigma_2)^{-1}}\chi_\lambda(g)\chi_T(g).
\end{align}

The set $\{(\sigma_1,\sigma_2)^{1}|\sigma_1,\sigma_2 \in H_{k-1}\}$ is the set of all elements of $H_{k-1} \times H_{k-1}$. Since $H_{k-1} \times H_{k-1} \subset S_{2^{k-1}}\times S_{2^{k-1}}$, we have
$$\chi_\lambda((\sigma_1,\sigma_2)^1)= \sum_{\mu,\nu \vdash 2^{k-1}}c_{\mu,\nu}^{\lambda}\chi_\mu(\sigma_1)\chi_\nu(\sigma_2).$$ 
We substitute this into the expression the first sum in Equation \eqref{eq:moment}
\begin{multline*}
\frac 1{|H_k|}\sum_{g=(\sigma_1,\sigma_2)^1}\chi_\lambda(g)\chi_T(g)=\\ \frac 1{|H_k|}\sum_{\sigma_1,\sigma_2 \in H_{k-1}}\;\sum_{\mu,\nu \vdash 2^{k-1}}c_{\mu,\nu}^{\lambda}\chi_\mu(\sigma_1)\chi_\nu(\sigma_2)\chi_T((\sigma_1,\sigma_2)^{1}).
\end{multline*}
Substituting the value of the character corresponding to $T$, and $|H_k|=2|H_{k-1}|^2$ into this equation, we have:
\begin{itemize}
\item
For $T$ of Type I or Type II:
\begin{align*}
\frac{1}{|H_k|}\sum_{g=(\sigma_1,\sigma_2)^1}\chi_\lambda(g)\chi_T(g)&=
\frac{1}{2}\sum_{\mu,\nu \vdash 2^{k}}r_{\mu T_1}r_{\nu T_1}.
\end{align*}
\item
For $T$ of Type III:
\begin{align*}
\frac{1}{|H_k|}\sum_{g=(\sigma_1,\sigma_2)^1}\chi_\lambda(g)\chi_T(g)&=
\frac{1}{2}\sum_{\mu,\nu \vdash 2^{k}}r_{\mu T_1} r_{\nu T_2} +r_{\mu T_2} r_{\nu T_1}.
\end{align*}
\end{itemize}
It is known that $c^{\lambda}_{\mu,\nu}=c^{\lambda}_{\nu,\mu}$, and from Proposition \ref{prop:lrcoeff} we know that $c^{\lambda}_{\mu,\nu}=1$ when $\mu= (a_1|b_1) \subset \lambda$ and $\nu =(a-a_1|b-b_1-1)$ or $\nu =(a-a_1-1|b-b_1-)$. This identification imposes the condition $a_1 > \frac{a}{2}$. Recall also that $c^{\lambda}_{\mu,\mu}=1$ iff $\mu=\frac{\lambda}{2}$.  Accounting for this, we modify the summation as under:

\begin{align}
\label{eq:part1}
\frac{1}{|H_k|}\sum_{g=(\sigma_1,\sigma_2)^1}\chi_\lambda(g)\chi_T(g)=
\begin{cases}
A(T_1,\lambda)+\frac{1}{2}r_{\frac{\lambda}{2}T_1}^2 & T \text{ Type I /Type II},\\
B(T_1,T_2,\lambda)+r_{\frac{\lambda}{2}T_1}r_{\frac{\lambda}{2}T_2} &T \text{  Type III},
\end{cases}
\end{align}
where $A(T_1,\lambda)$ and $B(T_1,T_2,\lambda)$ are defined as in Equations \eqref{eq:A} and \eqref{eq:B} respectively. 

The second sum in Equation \eqref{eq:moment} may be simplified as under:
\begin{align*}
\frac{1}{|H_k|}\sum_{g=(\sigma_1,\sigma_2)^{-1}}\chi_\lambda(g)\chi_T(g)=\frac{1}{|H_k|}\sum_{g=(\Id,\sigma)^{-1}}|[g]|\chi_\lambda(g)\chi_T(g),
\end{align*}
where the sum is now over representatives $\sigma$ of distinct conjugacy classes of $H_{k-1}$. Since $|(\Id,\sigma)^{-1}|= |H_{k-1}||[\sigma]|$:
\begin{align*}
\frac{1}{|H_k|}\sum_{g=(\sigma_1,\sigma_2)^{-1}}\chi_\lambda(g)\chi_T(g)&=\frac{1}{|H_k|}\sum_{g=(\Id,\sigma)^{-1}}|H_{k-1}||[\sigma]|\chi_\lambda(g)\chi_T(g)\\
&=\frac{1}{2|H_{k-1}|}\sum_{g=(\Id,\sigma)^{-1}}|[\sigma]|\chi_\lambda(g)\chi_T(g)\\
&=\frac{1}{2}\sum_{\sigma \in H_{k-1}}\frac{\chi_\lambda((\Id,\sigma)^{-1})\chi_T((\Id,\sigma)^{-1})}{|H_{k-1}|}.
\end{align*}
Let $\nu$ denote the cycle type of an element $(\Id,\sigma)^{-1}$ for $\sigma \in H_{k-1}$. Then by the Murnaghan-Nakayama rule we have:
\begin{displaymath}
\chi_\lambda((\Id,\sigma)^{-1})= \sum_{T \in \text{BST}(\lambda,\nu)}(-1)^{ht(T)}.
\end{displaymath}

It is easy to prove that if the cycle type of $\sigma$ is $(1^{k_0}2^{k_1}4^{k_2}\dotsc)$, then the cyle type of $(\Id,\sigma)^{-1}$ is $(2^{k_0}4^{k_1}8^{k_2}\dotsc)$. Thus, in particular, every tableau $T$ of content $\nu$ has $i$-strips (see Definition \ref{def:istrip}) of even size. Thus by Proposition \ref{prop:half} for $T \in \text{BST}(\lambda,\nu)$:
\begin{align*}
(-1)^{ht(T)}&=(-1)^{\text{vert}(\lambda)}(-1)^{ht(\psi(T))}\\
&=(-1)^{\text{vert}(\lambda)}\chi_{\frac{\lambda}{2}}(\sigma).
\end{align*}
Then we have:
\begin{align*}
\frac{1}{2}\sum_{\sigma \in H_{k-1}}\frac{\chi_\lambda((\Id,\sigma)^{-1})\chi_T((\Id,\sigma)^{-1})}{|H_{k-1}|}=\frac{(-1)^{\text{vert}(\lambda)}}{2}\sum_{\sigma \in H_{k-1}}\frac{\chi_{\frac{\lambda}{2}}(\sigma)\text{ext}(T)\chi_{T_1}(\sigma)}{|H_{k-1}|}
\end{align*}
Thus:
\begin{align}
\label{eq:part2}
\frac{1}{|H_k|}\sum_{g=(\sigma_1,\sigma_2)^{-1}}\chi_\lambda(g)\chi_T(g)=
\begin{cases}
 \frac{\text{ext}(T)(-1)^{\text{vert}(\lambda)}}{2}r_{\frac{\lambda}{2}T_1} & T \text{ Type I/Type II},\\
0 &T \text{ of Type III}.
\end{cases}
\end{align}
Adding Equations \eqref{eq:part1} and \eqref{eq:part2} proves the theorem.
\end{proof} 

The bijection between odd partitions and one-dimensional representations of Sylow 2-subgroups introduced in \cite{MR3687936} maps a hook partition $\lambda$ of $2^k$ to the unique one-dimensional representation of $H_k$ occuring with odd multiplicity, which we denote by $\eta(\lambda)$. We think of $\eta(\lambda)$ as a binary string of length $k$ under the encoding described in Definition \ref{def:beta_k}.
 
\begin{corollary}
\label{cor:gian_bij}
Given a hook partition $\lambda=(a|b)$ of size $2^k$, $k \geq 0$, the representation $\eta(\lambda)$ is the unique one-dimensional representation occuring in $R(\lambda)$ and it occurs with multiplicity one. Further:
\begin{align*}
\eta(\lambda)= 
\begin{cases}
0\eta(\frac{\lambda}{2}) & a \text{ odd }, b \equiv 0 \mod 4,\\
1\eta(\frac{\lambda}{2}) & a \text{ odd }, b \equiv 2 \mod 4,\\
1\eta(\frac{\lambda}{2}) & a \text{ even }, b \equiv 1 \mod 4,\\
0\eta(\frac{\lambda}{2}) & a \text{ even }, b \equiv 3 \mod 4.\\ 
\end{cases}
\end{align*}
\end{corollary}
\begin{proof}
We shall prove this inductively, and skip the case $k=1$ since it is an easy computation. If the result is true for all integers less than $k$, let $\chi_T$ be the unique one-dimensional representation occuring in $\lambda$. By the induction hypothesis,  $$r_{\mu T_1}r_{\nu T_1}=0$$
for $\mu \neq \nu$. Thus:
$$r_{\lambda T}= \frac{r_{\frac{\lambda}{2} T_1}^2}{2}+
\frac{\text{ext}(T)(-1)^{\text{vert}(\lambda)}}{2}
r_{\frac{\lambda}{2} T_1}.$$

Either $T=(r,\{\eta(\frac{\lambda}{2}),\eta(\frac{\lambda}{2})\})$ or $T=(r,\{\eta(\frac{\lambda}{2})\})$. It is the former if $\text{vert}(\lambda)$ is even and the latter when $\text{vert}(\lambda)$ is odd. With $\frac{\lambda}{2}=(a_1|b_1)$, either $(a|b)= (2a_1+1|2b_1)$ or $(a|b)=(2a_1|2b_1+1)$; the number of vertical dominoes in $\lambda$ is even when $a$ is odd and $b\equiv 0mod4$ or when $a$ is even and $b \equiv 3mod4$, and is odd otherwise.
\end{proof}
For an integer $n$ with $Bin(n)=\{k_1,\dotsc,k_s\}$ and an odd partition $\lambda$ of $n$, \cite[Lemma 1]{MR3510808} states that $\lambda$ has a unique hook $h_1$ of size $2^{k_1}$, and $\lambda/h_1$ is an odd partition. We apply this recursively to obtain a decomposition of $\lambda$ into the tuple $(h_1,\dotsc, h_s)$ of hook partitions $h_i$ of size $k_i$, for $i=1,\dotsc,s$. We call $(h_1,\dotsc, h_s)$ the hook decomposition of $\lambda$.

\begin{definition}
The restriction set of a partition $\mu$ of size $n$ is:  $$R(\mu)=\{F|\langle \chi_\mu,\chi_F \rangle_{P_n} >0\},$$ where $\chi_F$ is the irreducible character corresponding to a forest $F$ of size $n$. 
\end{definition}

\begin{prop}
\label{prop:rest_n}
For an odd partition $\lambda$ of $n$ with hook decomposition $(h_1,\dotsc, h_s)$, we have:
\begin{displaymath}
R(h_1)\times \dotsb \times R(h_s) \subset R(\lambda)
\end{displaymath} 
\end{prop}
\begin{proof}
Since $P_n \subset S_{2^{k_1}} \times \dotsb \times S_{2^{k_s}}$, we have $$\Res^{S_n}_{P_n}\chi_\lambda=\sum_{\mu_i \vdash 2^{k_i} \text{ for } 1,\dotsc,s}c^{\lambda}_{\mu_1,\mu_2}c^{\mu_2}_{\mu_3,\mu_4}\dotsc c^{\mu_{s-2}}_{\mu_{s-1},\mu_s}\prod_{i}\Res^{S_{2^{k_i}}}_{H_{k_i}}\chi_{\mu_i}.$$
Thus given a forest $F=(T_1,\dotsc,T_s)$ of size $n$,
\begin{displaymath}
\langle \chi_\lambda, \chi_F \rangle_{P_n}= \sum_{\mu_i \vdash 2^{k_i} \text{ for } 1,\dotsc,s}c^{\lambda}_{\mu_1,\mu_2}c^{\mu_2}_{\mu_3,\mu_4}\dotsc c^{\mu_{s-2}}_{\mu_{s-1},\mu_s}\prod_i r_{\mu_i T_i} .
\end{displaymath}
By Proposition \ref{prop:lr_n}, we know that the product $c^{\lambda}_{h_1,h_2}c^{h_2}_{h_3,h_4}\dotsc c^{h_{s-2}}_{h_{s-1},h_s}=1$. Thus for a forest $F=(T_1,\dotsc,T_s)$ where $T_i \in R(h_i)$, we see that $\langle \chi_\lambda, \chi_F \rangle_{P_n}>0$. 
\end{proof}

\section{Some generating functions}
\label{sec:gen_func}

In this section we find generating functions for conjugacy classes collected by class size and irreducible representations collected by dimension.

\begin{prop}
\label{th:gen_rep}
Let $a_{km}$ denote the number of irreducible representations of $H_k$ of dimension $2^m$. Let $g_k(t)= \sum_{m\geq 0}a_{km}t^m$. Then $g_k(t)$ satisfies the recurrence relation:

\begin{equation}
\label{eq:conjrec}
g_k(t)=\frac t2[g_{k-1}(t)^2-g_{k-1}(t^2)]+2g_{k-1}(t^2).
\end{equation}
\end{prop}

\begin{proof}
Note that by Definition \ref{def:bij_rep} we have $g_k(t)= \sum_{T} t^{log_2(\Dim(T))}$, where $T$ runs over all 1-2 binary trees of height $k$. Let $T^{\pm}$ denote the trees $(r,\{T,T\})$ and $(r,\{T\})$ respectively. From Remark \ref{rem:dim} we know that each of their dimensions satisfies the relation:$$t^{log_2(\Dim(T^{\pm})}=t^{2log_2(\Dim(T))}.$$
Thus each of the set of Type I and Type II trees are enumerated by $g_{k-1}(t^2)$ in the generating function.

Given two distinct trees $T_1$ and $T_2$ of height $k-1$ let $T=(r,\{T_1,T_2\})$. Then from Remark \ref{rem:dim}
$$t^{log_2(\Dim(T)}=t^{log_2(\Dim(T_1)\Dim(T_2))+1}.$$
The factor $\frac {g_{k-1}(t)^2-g_{k-1}(t^2)}{2}$ is the sum $\sum_{T}t^{log_2(\Dim(T_1)\Dim(T_2))}$ over all trees $T$ of Type III. This must then be multiplied by $t$ to enumerate such trees by their dimension.
\end{proof}
\begin{remark}
In particular $g_k(2)$ is the sum of the dimensions of all representations, which is equal to the number of involutions of $H_k$. $g_k(4)$ is the sum of the squares of dimensions of representations, thus $g_k(4)= 2^{2^k-1}$. Substituting into the equation, we have $$g_k(2)= g_{k-1}(2)^2+ g_{k-1}(4).$$ The involutions of $H_k$ are enumerated by the elements $(\sigma_1,\sigma_2)^1$ for involutions $\sigma_1$ and $\sigma_2$ of $H_{k-1}$, and by the elements $(\sigma,\sigma^{-1})^{-1}$ where $\sigma$ is an element of $H_{k-1}$. 
\end{remark}

\begin{prop}
\label{th:gen_conj}
Let $b_{km}$ be the number of conjugacy classes of $H_k$ of size $2^m$. Define the generating function $f_k(t):= \sum_{m \geq 0} b_{km}t^m$.
The ordinary generating function $f_k(t)$ satisfies the recurrence relation:

\begin{equation}
\label{eq:irrec}
f_k(t)=\frac t2[f_{k-1}(t)^2-f_{k-1}(t^2)]+f_{k-1}(t^2)+t^{2^{k-1}-1}f_{k-1}(t).
\end{equation}
\end{prop}

\begin{proof}
By Definition \ref{def:bij_con} we have $f_k(t)=\sum_T t^{log_2(\Or(T))}$. The sizes of conjugacy classes may be found in Table \ref{table:conj_class}. 

Let $T^{\pm}$ be as defined in the last proof. Since $\Or(T^+)=\Or(T)^2$, the monomials in the sum $f_k(t)$ corresponding to such trees are enumerated by the term $f_{k-1}(t^2)$. Since $\Or(T^{-}) = 2^{2^{k-1}-1}\Or(T)$, the monomials corresponding to such trees are enumerated by $t^{2^{k-1}-1}f_{k-1}(t)$.

Given two distinct trees $T_1$ and $T_2$ of height $k-1$, let $T=(r,\{T_1,T_2\})$. In this case we have $\Or(T)=2\Or(T_1)\Or(T_2)$, and as in the proof above, $\frac t2 g_{k-1}(t)^2-g_{k-1}(t^2)$ is an enumeration of the monomials corresponding to such trees. 
\end{proof}

The following corollary that generalises the above generating functions.

\begin{corollary}
Let $\alpha_{nm}$ denote the number of forests of size $n$ and dimension $m$ and let $\beta_{nm}$ denote the number of forests of size $n$ and order $m$. Define $G(v,t)= \sum_{n,m}\alpha_{nm}v^nt^m$ and $F(v,t)= \sum_{n,m}\beta_{nm}v^nt^m$. Then we have
\begin{align*}
G(v,t)=\prod_{i \geq 0} (1+g_{i}(t)v^{2^i})\\
F(v,t)=\prod_{i \geq 0} (1+f_{i}(t)v^{2^i}) 
\end{align*}
\end{corollary}
\section*{Acknowledgement}
I would like to thank my advisor Amritanshu Prasad for his guidance and advice. I am also grateful to Steven Spallone for his notes and for several helpful discussions.
\bibliography{notes-rewrite}
\bibliographystyle{plain}
\end{document}